\documentclass [11pt,a4paper]{amsart}

\usepackage{amssymb} 
\usepackage{amsfonts} 
\usepackage{amsmath}
\usepackage{amsthm} 
\usepackage{tikz}
\usepackage[utf8]{inputenc}
\usepackage[all]{xy}
\SelectTips{cm}{}
\usepackage{pifont}
\usepackage{enumitem}
\setlist{leftmargin=2\parindent}
\setlist[2]{leftmargin=1.5\parindent}
\setlist[3]{leftmargin=1\parindent}
\setlist[enumerate]{font=\normalfont}

\usepackage[pagebackref,
    ,pdfborder={0 0 0}
    ,urlcolor=black,a4paper,hypertexnames=false]{hyperref}
\hypersetup{pdfauthor={Stefan Friedl, Clara L\"oh},pdftitle={Epimorphism testing with virtually Abelian tragets}}

\usepackage{caption}

\newtheorem{lem}{Lemma}[section]
\newtheorem{thm}[lem]{Theorem}
\newtheorem{prop}[lem]{Proposition}
\newtheorem{cor}[lem]{Corollary} 

\theoremstyle{definition}
\newtheorem{defi}[lem]{Definition}
\newtheorem{setup}[lem]{Setup}
\newtheorem{conj}[lem]{Conjecture}
\newtheorem{question}[lem]{Question}

\newtheorem{exa}[lem]{Example}
\newtheorem{rem}[lem]{Remark} 

\DeclareMathOperator{\ab}{ab}
\DeclareMathOperator{\Span}{Span}
\DeclareMathOperator{\im}{im}
\DeclareMathOperator{\Hom}{Hom}

\DeclareMathOperator{\Aut}{Aut}
\DeclareMathOperator{\map}{map}
\DeclareMathOperator{\id}{id}
\DeclareMathOperator{\Free}{Free}

\newcommand{\N}{\ensuremath {\mathbb{N}}}

\newcommand{\Z} {\ensuremath {\mathbb{Z}}}

\newcommand{\EE}{\ensuremath{\mathbb{E}}}
\newcommand{\LL}{\ensuremath{\mathbb{L}}}
\newcommand{\KK}{\ensuremath{\mathbb{K}}}

\newcommand{\Lab}{\LL_{\ab}}
\newcommand{\Eab}{\EE_{\ab}}
\newcommand{\Kab}{\KK_{\ab}}

\DeclareMathOperator{\Chn}{C}

 \makeatletter
 \newcommand\norm{\bBigg@{0.8}}

 \makeatother
 \newcommand{\inparens}[2][flex]{\csname #1l\endcsname(#2%
                                 \csname #1r\endcsname)\mathclose{}}
 \newcommand{\inangles}[2][flex]{\csname #1l\endcsname\langle#2%
                                 \csname #1r\endcsname\rangle\mathclose{}} 
 \newcommand{\innorm}[2][flex]{\csname #1l\endcsname|#2%
                                 \csname #1r\endcsname|\mathclose{}}
 \newcommand{\genrel}[3][norm]{\csname #1l\endcsname\langle %
                              #2 %
                              \,\csname#1m\endcsname|\, %
                              #3 %
                              \csname#1r\endcsname\rangle\mathclose{}}
\def\args{\cdot}

\def\fa#1{%
  \forall_{#1}\;\;\;}

\title[Epimorphism testing with virtually Abelian targets]
      {Epimorphism testing\\ with virtually Abelian targets}

\author{Stefan Friedl}      
\author{Clara L\"oh}



\date{\today.\ \copyright{\ S.~Friedl, C.~L\"oh 2020}. 
    This work was supported by the CRC~1085 \emph{Higher Invariants} 
    (Universit\"at Regensburg, funded by the DFG)}

\begin{document}

\begin{abstract}
  We show that the epimorphism problem is solvable for targets that
  are virtually cyclic or a product of an Abelian group and a finite
  group.
\end{abstract}
\maketitle

\section{Introduction}

Given two finitely presented groups~$\Gamma$ and $\Lambda$, it is
natural to wonder if one can determine algorithmically whether there
exists a group epimorphism~$\Gamma \longrightarrow \Lambda$ or
not. For example, the existence of a group epimorphism~$\pi_1(M)
\longrightarrow \pi_1(N)$ between the fundamental groups of oriented
closed connected manifolds~$M$ and $N$ is a first, rudimentary,
necessary condition for the existence of a continuous map~$M
\longrightarrow N$ of degree $\pm 1$, see~\cite[p.~178]{loeh}.

If the domain group $\Gamma$ is trivial, then the epimorphism problem is
equivalent to deciding whether the target~$\Lambda$ is trivial or
not. However, it is well known that the triviality problem is
undecidable~\cite{miller}. Therefore, in general, the epimorphism
problem is undecidable.

Thus it is reasonable to restrict oneself to suitable classes of
groups. In this paper we want to study the following uniform version
of the epimorphism problem.

\begin{question}[uniform epimorphism problem]\label{q:uepi}
  Let $C$ and $D$ be two classes of finitely presented groups.  Does
  there exist an algorithm that solves the \emph{uniform epimorphism
    problem from $C$ onto $D$}\;?  More precisely, does there exist an
  algorithm that takes as an input a finite presentation~$\genrel SR$
  of a group in $C$ and a finite presentation $\genrel {S'}{R'}$ of a
  group in $D$ and that determines whether there exists an
  epimorphism~$\bigl| \genrel SR \bigr| \longrightarrow \bigl| \genrel
  {S'}{R'} \bigr|$ or not. When $C$ is the class of all finitely
  presented groups, then we refer to the above as the \emph{uniform
    epimorphism problem onto $D$}.
\end{question}

\begin{exa}
  \hfil
  \begin{itemize}
  \item As discussed above, the uniform epimorphism problem is
    undecidable whenever $C$ contains the trivial group and $D$ equals
    the class of all finitely presented groups.
  \item If $C=D$ is contained in the class of Hopfian groups, then the
    epimorphism problem is equivalent to the isomorphism problem for
    $C=D$.
  \item Let $C=D$ be the class of all finitely presented nilpotent
    groups. Remeslennikov showed that the uniform epimorphism problem
    from $C$ to $D$ is \emph{not} decidable~\cite{remeslennikov}. In
    particular the uniform nilpotency problem onto $C$ is undecidable.
    The proof by Remeslennikov is based on a reduction to the
    unsolvability of Hilbert's tenth problem.
\end{itemize}
\end{exa}

In contrast note that if $D$ is the class of finite groups or the
class of Abelian groups, then the uniform epimorphism problem onto $D$
is actually solvable. Indeed, for finite targets, we can compute the
whole set of epimorphisms (Proposition~\ref{prop:finepilist}), for
Abelian targets, we can explicitly compute the abelianisation of the
domain group and apply the structure theory of finitely generated
Abelian groups.

It is natural to ask whether the uniform epimorphism problem is
solvable for classes of groups that are ``close'' to being Abelian or
finite. This leads us to the following question.

\begin{question}\label{q:epi-onto-virtually-abelian}
Is the uniform epimorphism problem onto the class of virtually Abelian groups solvable?
\end{question}

Our main result gives a partial answer to
Question~\ref{q:epi-onto-virtually-abelian}.

\begin{thm}\label{mainthm}
  \hfil
  \begin{itemize}
  \item The uniform epimorphism problem onto the class of groups that
    are a direct product of a finitely generated Abelian group with a
    finite group is solvable $($see Theorem~\ref{thm:prod} for the
    precise formulation$)$.
  \item
    The uniform epimorphism problem onto the class of groups that are
    virtually cyclic is solvable $($see Theorem~\ref{thm:vZ} for the
    precise formulation$)$.
\end{itemize}
\end{thm}

As the epimorphism problem for virtually Abelian targets leads to a
similar problem in linear algebra (Section~\ref{subsec:translationLA})
as in the above case of nilpotent targets studied by
Remeslennikov~\cite{remeslennikov}, the following is plausible:

\begin{conj}
  The uniform epimorphism problem onto the class of all finitely generated
  virtually Abelian groups is \emph{not} decidable.
\end{conj}

We also have two questions on the two opposite rays of Bridson's
universe of groups.

\begin{question}
  \hfil
  \begin{itemize}
  \item 
    Let $D$ be the class of groups that are isomorphic to subgroups of
    products of a finitely generated Abelian group and a finite
    group. Is the uniform epimorphism problem onto $D$ solvable?
  \item Let $H$ be the class of hyperbolic groups. Is the uniform
    epimorphism problem from $H$ onto $H$ solvable?
\end{itemize}
\end{question}

We conclude this introduction with a discussion on the relevance of
the epimorphism problem for the isomorphism problem. To do so we
introduce the following notation: given a group $G$ and a class of
groups $C$ we define $\operatorname{Epi}(G,C)$ to be the class of
quotients of $G$ that lie in $C$.  If two groups $G$ and $H$ are
isomorphic, then for any $C$ we evidently have
$\operatorname{Epi}(G,C)=\operatorname{Epi}(H,C)$.

The set $C=\operatorname{Fin}$ of all isomorphism types of finite
groups is well-studied.  In fact, we know that for two finitely
generated groups $G$ and $H$ we have
$\operatorname{Epi}(G,\operatorname{Fin})=\operatorname{Epi}(G,\operatorname{Fin})$
if and only if the profinite completions of the two groups are
isomorphic~\cite[Corollary~3.2.8]{ribes-zalesskii}.

In practice studying $\operatorname{Epi}(-,\operatorname{Fin})$ can be
a very effective way for showing that two groups are not
isomorphic. For example this approach gets used in the tabulation of
knots.

Nonetheless there are even very basic examples where this approach
fails.  For example Baumslag \cite{baumslag} gave examples of pairs of
non-isomorphic groups $G$ and $H$ which are both virtually-$\Z$, in fact they are
both finite cyclic-by-$\Z$, but with
$\operatorname{Epi}(G,\operatorname{Fin})=\operatorname{Epi}(H,\operatorname{Fin})$.
This shows that it is useful to have larger classes of groups with
which we can probe $G$ and $H$.  Let $\operatorname{Virt-}\Z$ be the
class of groups that are virtually $\Z$. It is not difficult to give
an algorithm that lists all isomorphism types
in~$\operatorname{Virt-}\Z$.  This observation together with standard
algorithms and Theorem~\ref{mainthm}~(2) implies that there exists an
algorithm that can determine, given two finitely presented groups $G$
and $H$, whether or not
$\operatorname{Epi}(G,\operatorname{Virt-}\Z)=\operatorname{Epi}(H,\operatorname{Virt-}\Z)$.

\subsection*{A note on algorithms}

In this article, we describe algorithms in natural language. On the
one hand, these pseudo-algorithms lack some concreteness. On the other
hand, these descriptions have the advantage that we do not impose a
programming paradigm (such as declarative, imperative, functional,
\dots) and that we do not clutter the algorithmic ideas with
irrelevant technical details.

Moreover, as all of the algorithms that we consider will have
ridiculous worst-case complexity anyway, we do not pay any attention
on efficiency.

\subsection*{Organisation of this article}

After a short explanation of basic notation
(Section~\ref{sec:prelim}), we begin with characterisations of
existence of epimorphisms and translations of these characterisations
into linear algebra (Section~\ref{sec:prep}). In
Section~\ref{sec:linalg}, we explain how to solve the linear algebraic
problems in specific cases.  Using these methods, we solve the
epimorphism problem for targets that are products of Abelian and
finite groups as well as for virtually cyclic targets
(Theorem~\ref{thm:prod}, Theorem~\ref{thm:vZ}).

\subsection*{Acknowledgments.}
We wish to thank
  Nicolaus Heuer, Christoforos Neofytidis,
  Jos\'e Pedro Quintanilha
and Alan Reid for helpful comments and conversations.
We would like to thank the anonymous referee for spotting a mistake
in one of the cases of the proof of Proposition~\ref{prop:cgvZ}.

\section{Preliminaries}\label{sec:prelim}

Here, we fix basic notation and terminology.

\subsection{Generators and relations}

A \emph{group presentation} is a pair~$\genrel SR$ consisting of a
set~$S$ and a subset~$R$ of the free group~$\Free(S)$, freely
generated by~$S$. We will usually view~$\Free(S)$ as the set of
reduced words in~$S \cup S^{-1}$. A group presentation~$\genrel SR$ is
finite if both $S$ and $R$ are finite. If $\genrel SR$ is a group
presentation, then we denote the group described by this presentation
by
\[ \bigl| \genrel SR\bigr| := \Free(S) / \langle R \rangle^{\triangleleft}_{\Free(S)}.
\]

\begin{defi}[symmetric presentation]
  A group presentation~$\genrel SR$ is \emph{symmetric} if the
  following hold:
  \begin{itemize}
  \item Each relation in~$R$ is a positive word in~$S$.
  \item For each~$s \in S$ there exists an~$s' \in S$ with~$ss' \in R$
    or $s's \in R$.
  \end{itemize}
  Here, a word is \emph{positive} with respect to a generating set if
  it only consists of positive powers of generators. 
  It should be noted that this notion of symmetry of a presentation
  is different from the notion of ``symmetrized sets of relators'' by Lyndon
  and Schupp.
\end{defi}

\subsection{Virtually Abelian groups}

A group is \emph{virtually Abelian} if it contains a finite index
subgroup that is Abelian.

\begin{prop}\label{prop:vAbbasics}
  Let $\Lambda$ be a finitely generated virtually Abelian group. Then
  there exists a short exact sequence \textup{(}of groups\textup{)} of
  the form
  \[ \xymatrix{%
      1 \ar[r]
      & A \ar[r]
      & \Lambda \ar[r]
      & F \ar[r]
      & 1
    },
  \]
  where $F$ is a finite group and $A \cong \Z^d$ for some~$d \in \N$.
  In particular, $\Lambda$ has a finite presentation and is residually
  finite.
\end{prop}

\begin{proof}
  Let $\Lambda$ be a finitely generated virtually Abelian group. We
  only need to show that $\Lambda$ contains a finite-index finitely
  generated free Abelian normal subgroup. By hypothesis, $\Lambda$
  contains a finite-index Abelian normal subgroup~$B$. Since $\Lambda$
  is finitely generated we see that $B$ is also finitely
  generated. The normal core of a finite index free Abelian
  subgroup of~$B$ now has all the required
  properties.
\end{proof}

In Proposition~\ref{prop:vAbconstructive}, we will explain how a 
constructive description of virtually Abelian groups can be obtained
from any finite presentation. 

\section{Characterisations of existence of epimorphisms}\label{sec:prep}

\begin{setup}\label{set:epi}
  Let $\genrel SR$ be a symmetric finite presentation, let $\Gamma :=
  \bigl| \genrel SR\bigr|$ be the corresponding group, and let
  $\Lambda$ be a finitely generated virtually Abelian group, fitting
  into a short exact sequence
  \[ \xymatrix{%
      1 \ar[r]
      & A \ar[r]^{i}
      & \Lambda \ar[r]^{\pi}
      & F \ar[r]
      & 1
    },
  \]
  where $F$ is a finite group, $A \cong_\Z \Z^d$, and $i$ is the inclusion
  of a subgroup.
\end{setup}

The key idea is to split the epimorphism problem~$\Gamma
\longrightarrow \Lambda$ into finitely many cases by
\begin{itemize}
\item first determining all epimorphisms~$\Gamma \longrightarrow F$,
\item and then checking for each epimorphism~$\varphi \colon \Gamma
  \longrightarrow F$ whether there exists an epimorphism~$\Gamma
  \longrightarrow \Lambda$ that induces~$\varphi$.
\end{itemize}

\subsection{An abstract characterisation}

\begin{defi}[lifting/epimorphism set]\label{def:LEK}
  In the situation of Setup~\ref{set:epi}, let $\varphi \colon \Gamma
  \longrightarrow F$ be an epimorphism and let $K := \ker \varphi
  \subset \Gamma$. Then we write (Figure~\ref{fig:basicdiag})
  \begin{align*}
    L(\varphi)
    & := \bigl\{ \widetilde\varphi \in \Hom(\Gamma,\Lambda) \bigm|
    \pi \circ \widetilde \varphi = \varphi
    \bigr\}
    \\
    K(\varphi)
    & := \bigl\{ \widetilde\varphi|_K \bigm| \widetilde\varphi \in L(\varphi) \bigr\}
    \\
    E(\varphi)
    & := \bigl\{ \psi \in K(\varphi) \bigm| \psi(K) = A \bigr\}.
  \end{align*}
\end{defi}

\begin{figure}
  \begin{align*}
    \xymatrix@C=4em@R=1.75em{%
      1 \ar[d]
      &
      1 \ar[d]
      \\
      K \ar[d] \ar@{-->}[r]^-{\widetilde\varphi|_K}
      &
      A \ar[d]^-{i}
      \\
      \Gamma \ar[dr]_-{\varphi} \ar@{-->}[r]^-{\widetilde \varphi}
      &
      \Lambda \ar[d]^-{\pi}
      \\
      &
      F \ar[d]
      \\
      &
      1
    }
  \end{align*}
 
  \caption{Some basic notation (Definition~\ref{def:LEK})}
  \label{fig:basicdiag}
\end{figure}

\begin{prop}\label{prop:phiepi}
  In the situation of Setup~\ref{set:epi}, let $\varphi \colon \Gamma
  \longrightarrow F$ be an epimorphism. Then there exists an
  epimorphism~$\widetilde \varphi \colon \Gamma \longrightarrow
  \Lambda$ with~$\pi \circ \widetilde \varphi = \varphi$ if and only
  if~$E(\varphi)$ is non-empty.
\end{prop}
\begin{proof}
  We begin with a simple observation: If $\widetilde \varphi \in
  L(\varphi)$ and $K = \ker \varphi$, then
  \[   \im(\widetilde \varphi) \cap A
     = \im(\widetilde \varphi) \cap \ker \pi
     = \widetilde \varphi(K).
  \]

  Hence, if $\widetilde \varphi \in L(\varphi)$ is an epimorphism, then
  $\widetilde \varphi(K) = A$.

  Conversely, suppose that~$A = \widetilde \varphi(K)$. If $y \in \Lambda$,
  then we can write $y$ in the form
  \[ y = y' \cdot a
  \]
  where $y' \in \im \widetilde \varphi$ and $a \in A$ (because $\pi
  \circ \widetilde \varphi = \varphi$ is surjective). Therefore, there
  exist~$x' \in \Gamma$ and $k \in K$ with $y' = \widetilde
  \varphi(x')$ and $a= \widetilde \varphi(k)$. In particular,
  \[ \widetilde \varphi(x' \cdot k)
     = y' \cdot a = y,
  \]
  which shows that $\widetilde \varphi(\Gamma) = \Lambda$.
\end{proof}

\begin{cor}\label{cor:abstractepi}
  In the situation of Setup~\ref{set:epi}, the following are
  equivalent:
  \begin{enumerate}
  \item There exists an epimorphism~$\Gamma \longrightarrow \Lambda$.
  \item There exists an epimorphism~$\varphi \colon \Gamma \longrightarrow F$
    with~$E(\varphi) \neq \emptyset$.
  \end{enumerate}
\end{cor}
\begin{proof}
  Let $\widetilde \varphi \colon \Gamma \longrightarrow \Lambda$ be an
  epimorphism.  Then $\varphi := \pi \circ \widetilde \varphi \colon
  \Gamma \longrightarrow F$ is an epimorphism and thus $E(\varphi)
  \neq \emptyset$ by Proposition~\ref{prop:phiepi}.

  Conversely, if $\varphi \colon \Gamma \longrightarrow F$
  satisfies~$E(\varphi) \neq \emptyset$, then
  Proposition~\ref{prop:phiepi} shows in particular that there exists
  an epimorphism~$\Gamma \longrightarrow \Lambda$.
\end{proof}

\subsection{Translation to linear algebra}\label{subsec:translationLA}

In view of Corollary~\ref{cor:abstractepi}, the epimorphism
problem for~$\Lambda$ basically reduces to checking whether for given
epimorphisms~$\varphi \colon \Gamma \longrightarrow F$, the
set~$E(\varphi)$ is non-empty or not. However, in general, the
set~$K(\varphi)$ is \emph{not} finite. Therefore, in order to be able
to handle~$E(\varphi)$ it will be useful to have an efficient
description/parametrisation of~$K(\varphi)$. We will now give such
a description in terms of (integral) linear algebra:

\begin{setup}\label{set:epic}
  In the situation of Setup~\ref{set:epi}, we add a choice of a set-theoretic
  section~$\sigma \colon F \longrightarrow \Lambda$ to our data.
\end{setup}


\begin{defi}\label{def:LEKc}
  In the situation of Setup~\ref{set:epic}, let $\varphi \colon \Gamma
  \longrightarrow F$ be an epimorphism, let $K :=\ker \varphi \subset
  \Gamma$, and let $T \subset \Free(S)$ be a finite set representing a
  generating set of~$K$.
  \begin{itemize}
  \item A map~$f \colon S \longrightarrow A$ is \emph{hom-like} if
    the map
    \begin{align*}
      \sigma * f \colon
      S & \longrightarrow \Lambda \\
      s & \longmapsto \sigma\bigl(\varphi(s)\bigr) \cdot f(s)
    \end{align*}
    induces a well-defined group homomorphism~$\Gamma = \bigl|\genrel
    SR\bigr| \longrightarrow \Lambda$. If~$f$ is hom-like, we denote
    this homomorphism by~$\psi_f \colon \Gamma \longrightarrow \Lambda$.
  \item
    We then set 
    \begin{align*}
      \LL(\varphi)
      & := \bigl\{ f \in \map(S,A)
      \bigm| \text{$f$ is hom-like}\bigr\}
      \\
      \KK(\varphi)
      & := \bigl\{ \psi_f|_T \in \map(T,A) \bigm| f \in \LL(\varphi)\bigr\}
      \\
      \EE(\varphi)
      & := \bigl\{ f \in \KK(\varphi) \bigm| \text{$f(T)$ generates~$A$} \bigr\}. 
    \end{align*}
  \end{itemize}
\end{defi}

The notion of being hom-like depends on both~$\varphi$ and $\sigma$.
A straightforward calculation shows that the notation in
Definition~\ref{def:LEKc} is just a translation of
Definition~\ref{def:LEK} into a more explicit framework:

\begin{rem}\label{rem:LEKvsLEKc}
  In the situation of Setup~\ref{set:epic}, 
  let $\varphi \colon \Gamma \longrightarrow F$ be an epimorphism, 
  let $K :=\ker \varphi \subset \Gamma$, and let $T \subset \Free(S)$
  be a finite set representing a generating set in~$K$. Then the
  diagram
  \[ \xymatrix@C=4em@R=1.75em{%
    L(\varphi)
    \ar[r]^-{\args|_S}
    \ar[d]_-{\args|_K}
    & \LL(\varphi)
    \ar[d]^-{\psi_{\args}|_T}
    \\
    K(\varphi)
    \ar[r]^-{\args|_T}
    & \KK(\varphi)
    \\
    E(\varphi)
    \ar[r]_-{\args|_T}
    \ar@{^{(}->}[u]
    & \EE(\varphi)
    \ar@{^{(}->}[u]
    }
  \]
  is commutative and all three horizontal maps are bijections.
\end{rem}

\begin{cor}\label{cor:concreteepi}
  In the situation of Setup~\ref{set:epic}, the following are equivalent:
  \begin{enumerate}
  \item There exists an epimorphism~$\Gamma\longrightarrow \Lambda$.
  \item There exists an epimorphism~$\varphi \colon \Gamma \longrightarrow F$
    with~$\EE(\varphi) \neq \emptyset$.
  \end{enumerate}
\end{cor}
\begin{proof}
  We only need to combine Corollary~\ref{cor:abstractepi} with the
  translation from Remark~\ref{rem:LEKvsLEKc}.
\end{proof}

Hence, the epimorphism problem for~$\Lambda$ reduces to deciding
whether, for a given epimorphism~$\varphi \colon \Gamma
\longrightarrow F$, the set~$\EE(\varphi)$ is non-empty or not (this
will be explained in full detail in the proofs of
Theorem~\ref{thm:prod} and Theorem~\ref{thm:vZ}).

\begin{prop}\label{prop:Kc}
  In the situation of Setup~\ref{set:epic}, let $\varphi \colon \Gamma
  \longrightarrow F$ be an epimorphism, let $K :=\ker \varphi \subset
  \Gamma$, and let $T \subset \Free(S)$ be a finite set representing a
  generating set in~$K$. Then
  \begin{align*}
    \LL(\varphi) & \subset \map(S,A) \\
    \KK(\varphi) & \subset \map(T,A)
  \end{align*}
  are empty or they are affine subspaces of the \textup{(}finitely
  generated free\textup{)} $\Z$-modules~$\map(S,A)$ and $\map(T,A)$,
  respectively \textup{(}with respect to the point-wise module
  structures\textup{)}.
\end{prop}
\begin{proof}
  Because the map~$\psi_{\args}|_T \colon \LL(\varphi) \longrightarrow
  \KK(\varphi)$ is $\Z$-linear and surjective (by construction), it
  suffices to show that $\LL(\varphi)$ is an affine subspace
  of~$\map(S,A)$ or empty.  By construction, we have
  \[ \LL(\varphi) = \bigcap_{r \in R} \LL(\varphi,r), 
  \]
  where
  \begin{align*}
    \LL(\varphi,r) & := \bigl\{ f \in \map(S,A)
     \bigm| \Free(\sigma * f)(r) = e \bigr\}
  \end{align*}
  (and $\Free(\sigma * f) \colon \Free(S) \longrightarrow \Lambda$
  denotes the unique group homomorphism extending the map~$\sigma *
  f$). Thus, it suffices to show that for each~$r \in R$, the
  set~$\LL(\varphi,r)$ is an affine subspace of~$\map(S,A)$. We will
  accomplish this by interpreting~$\LL(\varphi,r)$ as solution space
  of a suitable (inhomogeneous) $\Z$-linear equation.

  Let $r \in R$, say~$r = s_1 \dots s_m$ with~$s_1,\dots, s_m \in
  S$. Moreover, let
  \[ \lambda_j := \sigma\bigl(\varphi(s_j)\bigr)
  \]
  for each~$j \in \{1,\dots,m\}$; then $\pi(\lambda_1 \cdot \dots
  \cdot \lambda_m) = \varphi(r) = e$ and thus the product~$\lambda :=
  \lambda_1 \cdot \dots \cdot \lambda_m$ lies in~$A$.  We now proceed
  as follows:
  
    Let $f \in \map(S,A)$ and, for~$j \in \{1,\dots.m\}$, let
    \[ x_j := f(s_j).
    \]
    Because $A$ is a normal subgroup of~$\Lambda$, for each~$\lambda
    \in \Lambda$, the conjugation homomorphism~$C_\lambda \colon
    \Lambda \longrightarrow \Lambda$ by~$\lambda$ restricts to an
    automorphism of~$A$.  We then have
    \begin{align*}
      \Free(\sigma * f) (r)
      & = \lambda_1 \cdot x_1 \cdot \dots \cdot \lambda_m \cdot x_m
      \\
      & = C_{\lambda_1}(x_1) \cdot C_{\lambda_1\cdot \lambda_2}(x_2)
      \cdot \dots \cdot C_{\lambda_1 \cdot \dots \cdot \lambda_m}(x_m)
      \cdot \lambda_1 \cdot \dots \cdot \lambda_m.
    \end{align*}
    Hence, $f$ lies in~$\LL(\varphi,r)$ if and only if
    \[ e = C_{\lambda_1}(x_1) \cdot C_{\lambda_1\cdot \lambda_2}(x_2)
      \cdot \dots \cdot C_{\lambda_1 \cdot \dots \cdot \lambda_m}(x_m)
      \cdot \lambda_1 \cdot \dots \cdot \lambda_m.
    \]
    Because~$\lambda = \lambda_1 \cdot \dots \cdot \lambda_m \in A$
    and the $C_{\lambda_j}$ are automorphisms of~$A$, we can
    reformulate this condition equivalently as the additive linear
    inhomogeneous equation
    \begin{align}
      C_{\lambda_1}(x_1) + C_{\lambda_1\cdot\lambda_2}(x_2) + \dots + C_{\lambda_1\cdot \dots\cdot\lambda_m}(x_m)
       = - \lambda. \label{eq:lineq}
    \end{align}
    This shows that $\LL(\varphi,r)$ is the solution set of an
    (inhomogeneous) $\Z$-linear equation in the $A$-valued
    variables~``$f(s)$'' with~$s \in S$.  More precisely,
    $\LL(\varphi,r)$ is the preimage of~$-\lambda$ under the
    $\Z$-linear map
    \begin{align*}
      A^S & \longrightarrow A \\
      (y_s)_{s\in S} & \longmapsto
      \sum_{s \in S} \sum_{j \in \{1,\dots, m\} \text{ with $s_j = s$}} C_{\lambda_1 \cdot \dots \cdot \lambda_j}(y_s).
      \qedhere
    \end{align*}
\end{proof}

\begin{rem}\label{rem:computeK}
  In the situation of Proposition~\ref{prop:Kc}, a finite ``generating
  set'' for~$\KK(\varphi)$ can be computed from the given data, a
  basis of~$A$, and matrices for the conjugations~$C_{\sigma(x)}$
  with~$x \in F$. We just need to follow the proof of
  Proposition~\ref{prop:Kc}:
  \begin{itemize}
  \item Using a basis of~$A$, we can write down the corresponding
    ``dual'' bases of~$\map(S,A)$ and~$\map(T,A)$, respectively.
  \item For~$r \in R$, we check whether $\varphi(r) = e$ or not
    (using a multiplication table for~$F$).
    \begin{itemize}
    \item
      If $\varphi(r) \neq e$, then $\LL(\varphi,r) = \emptyset$.
    \item
      If $\varphi(r) = e$, then we can write down the corresponding
      inhomogeneous $\Z$-linear equation~\eqref{eq:lineq}
      for~$\LL(\varphi,r)$. Here, we use the fact that
      \[ C_{\lambda_j} = C_{\sigma(\pi(\lambda_j))} 
      \]
      holds for all~$j \in \{1,\dots,m\}$ (because $A$ is Abelian).
      Using the Smith normal form
      algorithm~\cite[Algorithm~2.4.14]{cohen} we can then determine
      whether $\LL(\varphi,r)$ is empty or not. In the latter case, we
      can compute a finite set~$X \subset \map(S,A)$ and a~$b \in
      \map(S,A)$ with
      \[ \LL(\varphi,r) = \Span_\Z X_r + b_r.
      \]
    \end{itemize}
  \item Note that the finite intersection~$\LL(\varphi) = \bigcap_{r
    \in R} \LL(\varphi,r)$ is just the solution set to an affine
    linear equation system. Thus once again using the Smith normal
    form algorithm~\cite[Algorithm~2.4.14]{cohen} we can compute the
    finite intersection~$\LL(\varphi) = \bigcap_{r \in R}
    \LL(\varphi,r)$ in the following sense: We compute whether this
    intersection is empty or not; in the latter case, we compute a
    finite set~$X \subset \map(S,A)$ and an offset~$b \in \map(S,A)$
    with
    \[ \LL(\varphi) = \Span_\Z X + b.
    \]
  \item Because the (surjective) linear map~$\psi_{\args}|_T \colon
    \LL(\varphi) \longrightarrow \KK(\varphi)$ admits an explicit
    description (e.g., by a matrix), we can also compute a
    corresponding description for~$\KK(\varphi)$.
  \end{itemize}
\end{rem}

\subsection{Product targets}

When dealing with targets that are a product of a finitely generated
free Abelian group and a finite group, the following alternative
description will be convenient:

\begin{setup}\label{set:ab}
  Let $\genrel SR$ be a symmetric finite presentation, let $\Gamma :=
  \bigl|\genrel SR\bigr|$ be the corresponding group, let $A$ be an
  Abelian group (in the applications, this will be~$\Z^d$), let $F$ be
  a finite group, and let
  \[ \Lambda := A \times F.
  \]
  We write~$i \colon A \longrightarrow \Lambda$ for the inclusion as
  first factor and $\pi \colon \Lambda \longrightarrow F$ for the
  projection onto the second factor.
\end{setup}

\begin{defi}\label{def:notationab}
  In the situation of Setup~\ref{set:ab}, let $\varphi \colon \Gamma
  \longrightarrow F$ be an epimorphism, let $K :=\ker \varphi \subset
  \Gamma$, and let $\kappa \colon K_{\ab} \longrightarrow
  \Gamma_{\ab}$ be the homomorphism on the Abelianisations induced by
  the inclusion~$K \longrightarrow \Gamma$. Then, we write
  (Figure~\ref{fig:basicdiagab})
  \begin{align*}
    \Lab (\varphi) & := \Hom(\Gamma_{\ab}, A)
    \\
    \Kab (\varphi) & := \bigl\{ f \circ \kappa \in \Hom(K_{\ab}, A) \bigm| f \in \Lab(\varphi) \bigr\}
    \\
    \Eab (\varphi) & := \bigl\{ f \in \Kab(\varphi) \bigm| f(K_{\ab}) = A\bigr\}.
  \end{align*}
\end{defi}

\begin{figure}
  \begin{align*}
    \xymatrix@C=4em@R=1.75em{%
      K \ar@{->>}[r]^-{\pi_K} \ar[d]
      &
      K_{\ab} \ar@{-->}[r]^-{f\circ \kappa} \ar[d]^-{\kappa}
      &
      A \ar[d]^-{i}
      \\
      \Gamma \ar@{->>}[r]^-{\pi_\Gamma}
      \ar@{-->}@/_3ex/[rr]_{(f \circ \pi_\Gamma,\varphi)}
      &
      \Gamma_{\ab} \ar@{-->}[ur]_-{f}
      &
      A \times F
    }
  \end{align*}

  \caption{Some basic notation in the product case (Definition~\ref{def:notationab})}
  \label{fig:basicdiagab}
\end{figure}

\begin{prop}\label{prop:KEprod}
  In the situation of Setup~\ref{set:ab}, the following are equivalent:
  \begin{enumerate}
  \item There exists an epimorphism~$\Gamma \longrightarrow A \times F$.
  \item There exists an epimorphism~$\varphi \colon \Gamma \longrightarrow F$
    with~$\Eab(\varphi) \neq \emptyset$.
  \end{enumerate}
\end{prop}
\begin{proof}
  We consider the following commutative diagram:
  \[ \xymatrix@C=4em@R=1.75em{%
    L(\varphi)
    \ar@{<-}[r]^-{\text{\ding{192}}}
    \ar[d]_-{\args|_K}
    & \Lab(\varphi)
    \ar[d]^-{\args \circ \kappa}
    \\
    K(\varphi)
    \ar@{<-}[r]^-{\text{\ding{193}}}
    & \Kab(\varphi)
    \\
    E(\varphi)
    \ar@{<-}[r]_-{\text{\ding{194}}}
    \ar@{^{(}->}[u]
    & \Eab(\varphi)
    \ar@{^{(}->}[u]
    }
  \]
  Let $\pi_\Gamma \colon \Gamma \longrightarrow \Gamma_{\ab}$
  and $\pi_K \colon \Gamma \longrightarrow K_{\ab}$ denote the
  canonical projections. Because the target group splits as
  a product~$A \times F$, homomorphisms to the target
  split into pairs of homomorphisms to~$A$ and~$F$, respectively.
  Thus, we define
  \begin{align*}
    \text{\ding{192}} \colon \Lab(\varphi) & \longrightarrow L(\varphi)
    \\
    f & \longmapsto (f \circ \pi_\Gamma, \varphi)
    \\
    \text{\ding{193}} \colon \Kab(\varphi) & \longrightarrow K(\varphi)
    \\
    f & \longmapsto (f \circ \pi_K,e)
    \\
    \text{\ding{194}} \colon \Eab(\varphi) &\longrightarrow E(\varphi)
    \\
    f & \longmapsto (f \circ \pi_K,e).
  \end{align*}
  Then the above diagram is commutative and the horizontal maps are
  bijections (by definition of the various homomorphism
  sets). Therefore, the claim follows from the corresponding statement
  on the left column (Corollary~\ref{cor:abstractepi}).
\end{proof}

\section{Solving the problems in linear algebra}\label{sec:linalg}

In view of Corollary~\ref{cor:concreteepi}, Proposition~\ref{prop:Kc}
and Remark~\ref{rem:computeK}, we are interested in solving the
following problem:

\begin{question}[column-generation problem]
  Let $d, N \in \N$. Does there exist an algorithm that  
  given a finite subset~$X \subset M_{d \times N}(\Z)$, and
  a~$b \in M_{d \times N}(\Z)$ decides whether there exists
  an element in~$\Span_\Z X + b$ whose columns generate~$\Z^d$\;?
\end{question}

In general, such algorithms do \emph{not} exist~\cite{remeslennikov}.
However, as we will see, special cases of the column-generation
problem are solvable.

\subsection{The one-dimensional case}

For the treatment of virtually cyclic target groups, we will use the
following solution of the one-dimensional column-generation problem:

\begin{prop}\label{prop:cgvZ}
  Let $N \in \N$. Then there exists an algorithm that
  given a finite subset~$X \subset \Z^N$ and~$b \in \Z^N$
  decides whether there exists an element in~$\Span_\Z X + b$
  whose entries generate~$\Z$.
\end{prop}
\begin{proof}
  Let $X \subset \Z^N$ and let $b \in \Z^N$.  In view of the Smith
  normal form algorithm~\cite[Algorithm~2.4.14]{cohen}, we may assume
  without loss of generality that
  \[ X = \{ \alpha_1 \cdot e_1, \dots, \alpha_N \cdot e_N
         \},
  \]
  where $(e_1, \dots, e_N)$ is the standard basis of~$\Z^N$ and
  $\alpha_1,\dots, \alpha_N \in \Z$ satisfy
  \[ \alpha_1 \mid \alpha_2, \quad
     \alpha_2 \mid \alpha_3,\quad \dots,\quad
     \alpha_{N-1} \mid \alpha_N
  \]
  (note that these coefficients also can be zero).    
  Moreover, we write~$A := \Span_\Z X + b$
  and we denote the coefficients of~$b$ by~$b_1, \dots, b_N$.
  
  The original decision problem is then equivalent to deciding
  whether there exist~$x_1,\dots, x_N \in \Z$ with
  \[ \gcd( \alpha_1 \cdot x_1 + b_1, \dots, \alpha_N \cdot x_N + b_N) = 1.
  \]
  This problem can be solved as follows:
  \begin{itemize}
  \item If $\gcd(b_1,\dots, b_N) = 1$, then the answer is \emph{yes}
    (we can take~$x_1 = \dots = x_N = 0$).
  \item If $b = 0$, then:
    \begin{itemize}
    \item If $\alpha_1 = \pm 1$, then the answer is \emph{yes}
      (we can take~$x_1 =1$, $x_2 = \dots = x_N = 0$).
    \item If $\alpha_1 = 0$, then $\alpha_2 = \dots = \alpha_N = 0$
      and so~$A = \{0\} + 0 = \{0\}$. Hence, the answer is \emph{no}.
    \item If $\alpha_1 \not\in \{-1,0,1\}$, then $\gcd (\alpha_1,
      \dots, \alpha_N) > 1$, and so the answer is \emph{no}.
    \end{itemize}    
  \item If $b \neq 0$ and $c := \gcd(b_1, \dots, b_N) > 1$, then:
    \begin{itemize}
    \item If $\gcd(\alpha_1, c) > 1$, then the answer is \emph{no}
      because then we have also $\gcd(\alpha_1, \dots, \alpha_N , c)
      >1$ and so $\gcd(\alpha_1 \cdot x_1 + b_1, \dots, \alpha_N \cdot
      x_N + b_N) > 1$ for all~$x_1, \dots, x_N \in \Z$.
    \item If $\gcd(\alpha_1, c) = 1$, then:
      \begin{itemize}
      \item If there exists a~$j \in \{2,\dots, N\}$
        with~$b_j \neq 0$, then the answer is \emph{yes}:
        By Lemma~\ref{lem:d=1gcd} below, there exists an~$x \in \Z$
        such that
        \[ \gcd( x \cdot \alpha_1 + b_1, b_2, \dots, b_N) = 1.
        \]
        Hence, we can take~$x_1 = x, x_2 = \dots = x_N = 0$.
      \item If $b_j = 0$ for all~$j \in \{2,\dots, N\}$, then:
        \begin{itemize}
        \item If $N = 1$ or $\alpha_2 = 0$ (whence~$\alpha_2 = \dots = \alpha_N = 0$),
          then the answer is \emph{yes} if and only if $c$ is congruent to~$1$
          modulo~$\alpha_1$.
        \item If $\alpha_2 \neq 0$, then the answer is \emph{yes}:
          We can apply Lemma~\ref{lem:d=1gcd} below
          to~$\alpha_1, b_1, \alpha_2$ to find an~$x \in \Z$
          with
          \[ \gcd (x \cdot \alpha_1 + b _1, \alpha_2) =1.\]
          Hence, we can take~$x_1 = x$, $x_2 = 1$, $x_3 = \dots = x_N = 0$.
          \qedhere
        \end{itemize}
      \end{itemize}
    \end{itemize}
  \end{itemize}
\end{proof}

\begin{lem}\label{lem:d=1gcd}
  Let $N \in \N_{\geq 2}$, let $\alpha_1, b_1, \dots, b_N \in \Z$
  such that there exists a~$j \in \{2,\dots,N\}$ with~$b_j \neq 0$
  and 
  \begin{align*}
    \gcd(\alpha_1, b_1, \dots, b_N) & = 1.
  \end{align*}
  Then there exists an~$x \in \Z$ with
  \[ \gcd(x \cdot \alpha_1 + b_1 , b_2, \dots, b_N) = 1.
  \]
\end{lem}
\begin{proof}
  \emph{Assume} for a contradiction that for all~$x \in \Z$ we have
  \begin{align}\label{eq:cxassumption} 
    c_x := \gcd( x\cdot \alpha_1 + b_1, b_2, \dots, b_N) > 1.
  \end{align}
  Let $P \subset \N$ be the set of primes that divide~$\gcd(b_2, \dots, b_N)$
  (because~$b_j \neq 0$, this set is finite). Let $p \in P$.
  Then we claim that 
  there exists a~$d_p \in \Z$ with
  \[ D_p := \{ x \in \Z \mid \text{$p$ divides~$c_x$} \} \subset d_p + p \cdot \Z.
  \]
  Let us prove this claim:
  Let $x, y \in D_p$. It then suffices to show that $p$ divides~$x - y$
  (as then all elements of~$D_p$ share the same remainder modulo~$p$). 
  On the one hand, we have $p \mid x \cdot \alpha_1 + b_1$ and $p \mid
  y \cdot \alpha_1 + b_1$, and so
  \[ p \bigm| (x-y) \cdot \alpha_1.
  \]

  On the other hand, $p$ does \emph{not} divide~$\alpha_1$: Because of
  $p \mid c_x$, we have that $p \mid \gcd(b_2, \dots, b_N)$. If $p
  \mid \alpha_1$, then $p \mid b_1$ (because $p \mid x \cdot \alpha_1
  + b_1$) and so $p \mid \gcd(\alpha_1, b_1, \dots, b_N)$, which
  contradicts the assumptions in the proposition. therefore, $p$ does
  not divide~$\alpha_1$.

  Because $p$ is prime, it follows that $p \mid (x-y)$.
  This proves the claim.

  By assumption~\eqref{eq:cxassumption}, the construction of the
  sets~$D_p$, and the fact that each~$c_x$ has prime factors in~$P$,
  we have
  \[ \Z = \bigcup_{p \in P} D_p
        \subset \bigcup_{p \in P} (d_p + p \cdot \Z).
  \]
  Because $P$ is finite, the Chinese remainder theorem yields an
  element~$z \in \Z$ such that for all~$p \in P$ we have
  \[ z \not\equiv d_p \mod p.
  \]
  But this contradicts~$\Z = \bigcup_{p \in P} D_p$. Hence, we can
  conclude that there must be an~$x \in \Z$ with~$c_x = 1$.
\end{proof}

\subsection{The symmetric homogeneous case}

Let us now turn to the situation of target groups that decompose as a
product~$A \times F$ of a finitely generated free Abelian group~$A$
and a finite group~$F$. In this case, in the situation of
Proposition~\ref{prop:Kc}, the subset~$\KK(\varphi) \subset \map(T,A)$
is linear subspace (and not only an affine subspace) and the
equation~\eqref{eq:lineq} is invariant under automorphisms of~$A$
(because all the conjugations~$C_{\dots}$ are just the identity map).
Therefore, we end up with a very special version of the
column-generation problem (which turns out to be solvable).

However, instead of using the notation of the column-generation
problem (which is rather confusing in this case), we prefer to use an
alternative description of~$\KK(\varphi)$ and $\EE(\varphi)$, which is
more convenient (see
Proposition~\ref{prop:KEprod}). Proposition~\ref{prop:cgprod} will
then enter in the proof of Theorem~\ref{thm:prod}.

Every finitely generated $\Z$-module is finitely presented (because
$\Z$ is Noetherian) and can thus be described by a matrix over~$\Z$.
If $m,n \in \N$ and~$A \in M_{n \times m}(\Z)$, then we write
\[ M(A) := \Z^n/ \{ A \cdot x \mid x \in \Z^m \}
\]
for the finitely generated $\Z$-module presented by~$A$. 

\begin{prop}\label{prop:cgprod}
  There exists an algorithm that, given the input
  \begin{itemize}
  \item matrices~$A_1 \in M_{n_1 \times m_1}(\Z)$ and $A_2 \in M_{n_2
    \times m_2}(\Z)$,
  \item a homomorphism~$\kappa \colon M(A_1) \longrightarrow M(A_2)$
    \textup{(}given as an $n_2 \times n_1$-matrix\textup{)},
  \item a~$d\in \Z$,
  \end{itemize}
  decides whether there exists an epimorphism~$\psi \colon M(A_1)
  \longrightarrow \Z^d$ such that there is a homomorphism~$\widetilde
  \psi \colon M(A_2) \longrightarrow \Z^d$ with~$\widetilde \psi \circ
  \kappa = \psi$.
  \[ \xymatrix@C=4em@R=1.75em{%
    M(A_1) \ar@{->>}[r]^-{\psi} \ar[d]_-{\kappa}
    & \Z^d
    \\
    M(A_2) \ar@{-->}[ur]_-{\widetilde \psi}
    &
    }
  \]
\end{prop}
\begin{proof}
  Using the Smith normal form of~$A_1$ and $A_2$, we can compute the
  maximal free quotients of~$M(A_1)$ and $M(A_2)$ and the
  corresponding contribution of~$\kappa$. Because the target
  group~$\Z^d$ is free Abelian, we can therefore assume without loss
  of generality that $M(A_1)$ and $M(A_2)$ are free. Moreover, the
  image of~$\kappa$ can be computed and so we may assume that $\kappa$
  is the inclusion of a submodule (of which we know a basis in Smith
  normal form).

  Hence, we reduced the original problem to the following decision
  problem: Given the input
  \begin{itemize}
  \item $d, n \in \N$,
  \item $\alpha_1, \dots, \alpha_n \in \Z$ with $\alpha_1 \mid
    \alpha_2, \dots, \alpha_{n-1} \mid \alpha_n$,
  \end{itemize}
  decide whether there exists an epimorphism~$\psi \colon N  
  \longrightarrow \Z^d$ that admits an extension to a
  homomorphism~$\Z^n \longrightarrow \Z^d$, where
  \[ N := \Span_\Z\{\alpha_1\cdot e_1, \dots, \alpha_n \cdot e_n\} \subset \Z^n
  \]

  This problem can be solved as follows: Let $r \in \{0,\dots,n\}$ be
  the minimal index for which~$\alpha_{r+1} \not \in \{1,-1\}$ (where
  we set $r := n$ if $\alpha_n \in \{1,-1\}$). We then distinguish the
  following cases:
  \begin{itemize}
  \item If $r \geq d$, then the answer is \emph{yes}: clearly, the
    projection
    \begin{align*}
      \psi \colon \Z^n & \longrightarrow \Z^d
      \\
      x & \longmapsto (x_1,\dots, x_d)
    \end{align*}
    onto the first $d$ coordinates restricts to an epimorphism~$N
    \longrightarrow \Z^d$.
  \item If $r < d$, then the answer is \emph{no}:
    \begin{itemize}
    \item If $n < d$ or $\alpha_{r+1} = 0$, then the rank of~$N$ is
      smaller than~$d$ (and so there does not exist any epimorphism~$N
      \longrightarrow \Z^d$).
    \item If $n \geq d$ and $\alpha_{r+1} \neq 0$, then: Let
      $\widetilde \psi \colon \Z^n \longrightarrow \Z^d$ be a
      homomorphism and let $\psi := \widetilde \psi|_N$. We will now
      show that $\psi$ is \emph{not} surjective:

      Let $N' := \psi(N)$ and let $p$ be a prime factor of~$\alpha_{r+1}$.
      Then $N = \Z^r \oplus p \cdot R$ for some submodule~$R \subset \Z^{n-r}$
      and hence
      \[ N' \otimes_{\Z} \Z/p
      = \bigl( \psi(\Z^r) + p \cdot \psi(R)\bigr) \otimes_\Z \Z/p
      = \psi(\Z^r) \otimes_\Z \Z/p
      \]
      has $\Z/p$-dimension at most~$r$, which in turn is smaller than~$d$.
      In particular, $\psi$ cannot be surjective.
      \qedhere
    \end{itemize}
  \end{itemize}
\end{proof}

\section{Solvability}\label{sec:solv}

Using the tools of Section~\ref{sec:prep} and
Section~\ref{sec:linalg}, we will now establish solvability of the
epimorphism problem for targets that are products of Abelian groups
with finite groups or targets that are virtually cyclic
(Theorem~\ref{thm:prod} and Theorem~\ref{thm:vZ}).

\subsection{Preparation}

For the sake of completeness, we recall some basics from algorithmic
group theory.


\begin{prop}\label{prop:symmgenrel}
  There exists an algorithm that, given the input
  \begin{itemize}
  \item a finite presentation~$\genrel SR$,
  \end{itemize}
  determines a symmetric finite presentation~$\genrel {S'} {R'}$
  with~$S \subset S'$ such that the inclusion~$S \longrightarrow S'$
  induces an isomorphism~$\bigl| \genrel SR\bigr| \longrightarrow
  \bigl|\genrel{S'}{R'}\bigr|$.
\end{prop}
\begin{proof}
  For instance, we can take
  \begin{align*}
    S' := & \; S \times \{-1,1\}
    \\
    R' := & \; \bigl\{ (s,1) (s,-1) \bigm| s \in S \bigr\}
    \\
    \cup &\; \bigl\{ (s_1,\varepsilon_1) \cdots (s_m,\varepsilon_m)
    \bigm| m \in \N,\ s_1,\dots, s_m \in S,\ \varepsilon_1,\dots, \varepsilon_m \in \{-1,1\},
    \\
    &\; \phantom{\bigl\{ (s_1,\varepsilon_1) \cdots (s_m,\varepsilon_m)
    \bigm|} 
    \ s_1^{\varepsilon_1} \cdots s_m^{\varepsilon_m} \in R 
    \bigr\}.
  \end{align*}
  Strictly speaking, in this construction, $S$ is not a subset of~$S'$,
  but this can be fixed by renaming.
\end{proof}

\begin{prop}\label{prop:finepilist}
  There exists an algorithm that, given the input
  \begin{itemize}
  \item a finite presentation~$\genrel SR$,
  \item a finite group~$F$ \textup{(}as set of elements and its
    multiplication table\textup{)},
  \end{itemize}
  determines the set of all epimorphisms~$\bigl|\genrel SR\bigr|
  \longrightarrow F$.
\end{prop}
\begin{proof}
  Because $S$ is a generating set of~$\Gamma := \bigl|\genrel
  SR\bigr|$, group homomorphisms~$\Gamma \longrightarrow F$ can be
  represented by maps~$S \longrightarrow F$. We will compute the set
  of all epimorphisms~$\Gamma \longrightarrow F$ in the sense that we
  compute the subset of~$\map(S,F)$ consisting of all maps
  corresponding to epimorphisms~$\Gamma \longrightarrow F$.
  
  Let $n := |S|$ and let $S = \{s_1,\dots, s_n\}$.
  \begin{itemize}
  \item Then we can compute the finite set~$\map(S,F)$.
  \item For each~$f \in \map(S,F)$, we can check whether for
    each~$r \in R$, we have $\Free(f)(r) = e$ in~$F$ (using
    the multiplication table in~$F$). 
    Hence, we can compute the finite set
    \[ H := \bigl\{ f \in \map(S,F) \bigm| \fa{r \in R} \Free(f) = e \text{ in~$F$}\bigr\}, 
    \]
    which corresponds to the set of all homomorphisms~$\Gamma
    \longrightarrow F$ (via the universal property of generators and
    relations).
  \item
    We then compute the set of all generating sets of~$F$ as follows:
    For each subset~$T \subset F$, the set
    \[ G(T) := \bigl\{ t_1 \cdot \dots \cdot t_n
       \bigm| n \in \{0,\dots, |F|\},\ t_1, \dots, t_n \in T \cup T^{-1} \bigr\} \subset F
    \]
    equals the subgroup of~$F$ generated by~$T$ (by the pigeon-hole
    principle, longer words in~$T \cup T^{-1}$ cannot contribute new
    elements).

    Therefore, we can compute the set of all generating sets of~$F$ as
    the (finite) set
    \[ G := \{ T \subset F \mid G(T) = F \}.
    \]
  \item
    Hence, we can compute
    \[ E := \bigl\{ f \in H \bigm| f(S) \in G \bigr\}, 
    \]
    which corresponds to the set of all epimorphisms~$\Gamma
    \longrightarrow F$.  \qedhere
  \end{itemize}
\end{proof}

\begin{prop}\label{prop:finindex}
  There exists an algorithm that, given the input
  \begin{itemize}
  \item a symmetric finite presentation~$\genrel SR$,
  \item a finite group~$F$ \textup{(}as set of elements and its
    multiplication table\textup{)},
  \item an epimorphism~$\varphi \colon \bigl|\genrel SR\bigr|
    \longrightarrow F$ \textup{(}given by the images on~$S$\textup{)},
  \end{itemize}
  determines a symmetric finite presentation of~$\ker \varphi$
  \textup{(}where the generators are specified as words
  in~$S$\textup{)}.
\end{prop}
\begin{proof}
  The given data allows to find a map~$\sigma \colon F \longrightarrow
  \Free(S)$ with
  \[ \varphi \circ \pi \circ \sigma = \id_F,
  \]
  where $\pi \colon \Free(S) \longrightarrow \bigl| \genrel SR \bigr|$
  denotes the canonical projection (by enumerating all elements
  in~$\Free(S)$ and computing their images in~$F$ via~$\varphi \circ
  \pi$, until a preimage is found for every element in~$F$); without
  loss of generality, we may assume that $\sigma(e) =\varepsilon$. In
  other words, $\sigma$ specifies a coset representative system
  for~$\ker \varphi$ in~$\bigl|\genrel SR\bigr|$ (expressed in terms
  of words over~$S$). We write
  \[ c:= \sigma \circ \varphi \circ \pi \colon \Free(S) \longrightarrow \Free(S)
  \]
  for the map that determines the coset representative of an element
  selected by~$\sigma$.  Then the words
  \[ \bigl\{ \sigma(f) \cdot s (c (\sigma(f) \cdot s) )^{-1} 
     \bigm| s \in S,\ f \in F \bigr\}
  \]
  describe a generating set of~$\ker \varphi$~\cite[Theorem~2.7]{mks}.
  The Reidemeister rewriting process associated with respect to this
  generating set and the map~$c$ then computes a finite presentation
  of~$\ker \varphi$~\cite[Corollary~2.7.2, Theorem~2.8]{mks}. Finally,
  we symmetrise this finite presentation via
  Proposition~\ref{prop:symmgenrel}.
\end{proof}

For virtually Abelian targets that do not decompose as a product of a
free Abelian group and a finite group, we first want to clarify what
it means that a virtually Abelian group is given as
``input''. Naively, we could just take a finite presentation~$\genrel
SR$ of which we know for some external reason that the group~$\bigl|
\genrel SR\bigr|$ is virtually Abelian. A more constructive point of
view would require to include a reason why and how the given group is
virtually Abelian, i.e., that we are given a constructive description
of this group as extension of a finitely generated free Abelian group
by a finite group. In fact, every naive description can be turned
algorithmically into a constructive description. We will explain this
now in detail:

\begin{prop}\label{prop:vAbconstructive}
  There exists an algorithm that, given the input
  \begin{itemize}
  \item a finite presentation~$\genrel SR$ of a virtually Abelian
    group
  \end{itemize}
  determines
  \begin{itemize}
  \item a finite group~$F$ \textup{(}as set of elements and its
    multiplication table\textup{)},
  \item a $d \in \N$, 
  \item a group homomorphism~$C \colon F \longrightarrow \Aut(\Z^d)$,
  \item a cocycle~$c \in \Chn^2(F; \Z^d)$ \textup{(}with respect to
    the action~$C$\textup{)}
  \end{itemize}
  such that $\bigl| \genrel SR\bigr|$ is isomorphic to the extension
  group of~$\Z^d$ by~$F$ that corresponds to the cocycle~$c$.
\end{prop}

Before giving the proof, we briefly review the cocycle notation: Let
$F$ be a group, let $A$ be an Abelian group, and let $C \colon F
\longrightarrow \Aut(A)$ be an $F$-action on~$A$. Then $\Chn_2(G)$,
the bar resolution of~$G$ in degree~$2$ is the free $\Z F$-module,
freely generated by the pairs~$([g_1 | g_2])_{g_1, g_2 \in F}$.  Then
a \emph{cocycle}~$c \in \Chn^2(F;A)$ is a $\Z F$-linear map~$\Chn_2(G)
\longrightarrow A$ that satisfies the cocycle condition
\[ 0 
= C(g_1) \bigl(c [g_2 | g_3] \bigr)
- c [g_1 \cdot g_2 | g_3]
+ c [g_1 | g_2 \cdot g_3]
- c [g_1 | g_2]
\]
for all~$g_1, g_2, g_3 \in F$. 

Now let us recall the following explicit description of the extension
group~$\Lambda$ of~$\Z^d$ by~$F$ corresponding
to~$c$~\cite[Chapter~IV.3]{brown}: As underlying set, we take the
Cartesian product~$A \times F$ and as multiplication, we use
\begin{align*}
    (A \times F) \times (A \times F)
    & \longrightarrow (A \times F)
    \\
    \bigl((x,y), (x',y')\bigr)
    & \longmapsto
    \bigl( x + C(y)(x') + c(1 \cdot [y|y']), y \cdot y'\bigr);
\end{align*}
the neutral element of~$\Lambda$ is~$(e',e_F)$, where~$e' = - c(1
\cdot [e_F|e_F])$.  Then $\Lambda$ fits into the extension
  \[ \xymatrix{%
    1 \ar[r]
  & A \ar[r]^-{i}
  & \Lambda \ar[r]^-{\pi}
  & F \ar[r]
  & 1}
  \]
where
\begin{align*}
    i \colon A & \longrightarrow A \times F
    \\
    x & \longmapsto (x + e', e_F)
    \\
    \pi \colon \Lambda & \longrightarrow F
    \\
    (x,y) & \longrightarrow y.
\end{align*}
A set-theoretic section of~$\pi$ is, for instance,
\begin{align*}
    \sigma \colon F & \longrightarrow \Lambda \\
    y & \longmapsto (0,y).
\end{align*}

\begin{proof}[Proof of Proposition~\ref{prop:vAbconstructive}]
  Let $\Gamma := \bigl| \genrel SR \bigr|$.  We enumerate all
  (multiplication tables of isomorphism types of) finite groups (e.g.,
  as subgroups of finite permutation groups). For every finite
  group~$F$, we then perform the following steps:
  \begin{itemize}
  \item We compute the (finite) set of all epimorphisms~$\Gamma
    \longrightarrow F$ (using Proposition~\ref{prop:finepilist}).
  \item
    For each epimorphism~$f \colon \Gamma \longrightarrow F$, we
    compute a finite presentation~$\genrel TQ$ of~$\ker f$
    (Proposition~\ref{prop:finindex}), where the elements of~$T$ are
    specified as words in the given generating set~$S$ of~$\Gamma$.
  \item
    We then check whether all elements in~$T$ commute with each other.
    This is possible for the following reason: As finitely generated
    virtually Abelian group, $\Gamma$ is residually finite and
    therefore the uniform solution of the word problem for residually
    finite groups~ \cite[Theorem~5.2]{miller} provides us with an
    explicit algorithm to solve the word problem for~$\Gamma$ in the
    presentation~$\genrel SR$.
  \item
    If not all elements in~$T$ commute with each other, then $\ker f$
    is not Abelian and we discard~$f$ and proceed with the next
    epimorphism/finite group.
  \item
    If all elements in~$T$ commute with each other, then $\ker f$ is
    Abelian. From the presentation~$\genrel TQ$, we can compute (via
    the Smith normal form algorithm) whether $\ker f$ is free Abelian
    or not.
    \begin{itemize}
    \item If $\ker f$ is not free Abelian, then we discard~$f$ and
      proceed with the next epimorphism/finite group.
    \item If $\ker f$ is free Abelian, then we proceed as follows:
      \begin{itemize}
      \item
        We set~$A :=\ker f$ and we determine a basis~$B$ of~$A$ (via
        the standard algorithm). This also provides us with a
        transformation that allows to rewrite elements in~$T$ in this
        basis.
      \item 
        We then compute the $F$-conjugation action~$F \colon \Aut(A)$
        with respect to this basis: We search for $f$-lifts of each
        element in~$F$ and then compute for each of these lifts~$x$
        and each member~$b$ of~$B$ the conjugation~$x \cdot c \cdot
        x^{-1}$ in~$\Gamma$. The result will first be a word in~$S
        \cup S^{-1}$. We then rewrite this word in terms of~$T$ (via
        the Reidemeister rewriting process; proof of
        Proposition~\ref{prop:finindex}), and then in terms of~$B$.
      \item 
        Finally, we compute the cocycle~$c \in \Chn^2(F;A)$ of the
        extension
        \[ \xymatrix{%
          1 \ar[r]
          & A = \ker f \ar[r]
          & \Gamma \ar[r]^f
          & F \ar[r]
          & 1
        }
        \]
        through the well-known explicit formula~\cite[Chapter~IV.3]{brown}.
      \end{itemize}
    \end{itemize}
   \end{itemize}
  Because $\bigl|\genrel SR\bigr|$ is known to be virtually Abelian,
  Proposition~\ref{prop:vAbbasics} guarantees that this algorithm
  terminates.
\end{proof}

\subsection{Product targets}

\begin{thm}\label{thm:prod}
  There exists an algorithm that, given the input
  \begin{itemize}
  \item a finite presentation~$\genrel SR$,
  \item a finite group~$F$ \textup{(}as set of elements and its
    multiplication table\textup{)},
  \item a~$d\in\N$,
  \end{itemize}
  decides whether there exists an epimorphism~$\bigl|\genrel SR\bigr|
  \longrightarrow \Z^d \times F$ or not.
\end{thm}

\begin{proof}
  We write~$\Gamma := \bigl| \genrel SR\bigr|$. In view of
  Proposition~\ref{prop:KEprod}, it suffices to check for each
  epimorphism~$\varphi \colon \Gamma \longrightarrow F$,
  whether~$\Eab(\varphi) = \emptyset$ or not.

  \begin{itemize}
  \item By Proposition~\ref{prop:symmgenrel}, we may assume without
    loss of generality that $\genrel SR$ is a symmetric finite
    presentation.
  \item
    We compute the set~$E$ of all epimorphisms~$\Gamma \longrightarrow
    F$ (as a subset of~$\map(S,F)$;
    Proposition~\ref{prop:finepilist}).
  \item We determine a finite presentation of~$\Gamma_{\ab}$, namely
    \[ \genrel S {R \cup \{ sts^{-1}t^{-1}\mid s,t \in S \}}.
    \]
    For each~$\varphi \in E$, we determine finite presentations of~$K
    := \ker \varphi$ (Proposition~\ref{prop:finindex}) and then also
    of~$K_{\ab}$.  This allows us to find integral matrices~$A_1$ and
    $A_2$ with canonical isomorphisms~$K_{\ab} \cong_\Z M(A_1)$ and
    $\Gamma_{\ab} \cong_\Z M(A_2)$ as well as a matrix description of
    the corresponding homomorphism~$M(A_1) \longrightarrow M(A_2)$
    induced by the inclusion~$K \longrightarrow \Gamma$.
  \item
    We then compute (using these finite presentations in matrix form
    and Proposition~\ref{prop:cgprod}) the subset
    \[ \widetilde E := \{ \varphi \in E \bigm| \Eab(\varphi) \neq \emptyset \}
    \]
  \item
    \begin{itemize}
    \item
      If $\widetilde E \neq \emptyset$, then the answer is \emph{yes},
      there exists an epimorphism~$\Gamma \longrightarrow F$.
      
    \item
      If $\widetilde E = \emptyset$, the answer is \emph{no}.
    \end{itemize}
  \end{itemize}
  Correctness of this algorithm is guaranteed by
  Proposition~\ref{prop:KEprod}.
\end{proof}

\subsection{Virtually cyclic targets}

\begin{thm}\label{thm:vZ}
  There exists an algorithm that, given the input
  \begin{itemize}
  \item a finite presentation~$\genrel SR$, 
  \item a finite presentation~$\genrel {S'}{R'}$ of a group~$\Lambda$
    that is virtually~$\Z$,
  \end{itemize}
  decides whether there exists an epimorphism~$\bigl|\genrel SR\bigr|
  \longrightarrow \Lambda$ or not.
\end{thm}

\begin{proof}[Proof of Theorem~\ref{thm:vZ}]  
  We write $\Gamma := \bigl| \genrel SR \bigr|$.  In view of
  Corollary~\ref{cor:concreteepi}, it suffices to check for each
  epimorphism~$\varphi \colon \Gamma \longrightarrow F$,
  whether~$\EE(\varphi) = \emptyset$ or not.
  \begin{itemize}
  \item By Proposition~\ref{prop:symmgenrel}, we may assume without loss
    of generality that $\genrel SR$ is a symmetric finite presentation.
  \item
    Using the algorithm of Proposition~\ref{prop:vAbconstructive}, we
    transform the presentation~$\genrel{S'}{R'}$ into 
    \begin{itemize}
    \item a finite group~$F$ (as set of elements and its
      multiplication table),
    \item a group homomorphism~$C \colon F \longrightarrow \Aut(\Z)$,
    \item a cocycle~$c \in \Chn^2(F; \Z)$ (with respect to the action~$C$),
    \end{itemize}
    such that $\Lambda$ is isomorphic to the extension of~$\Z$ by~$F$
    corresponding to the cocycle~$c$.
  \item We compute the set~$E$ of all epimorphism~$\Gamma
    \longrightarrow F$ (as a subset of~$\map(S,F)$;
    Proposition~\ref{prop:finepilist}).
  \item We then compute (using the algorithm outlined in 
    Remark~\ref{rem:computeK}, and Proposition~\ref{prop:cgvZ}) the subset
    \[ \widetilde E := \bigl\{ \varphi \in E \bigm| \EE(\varphi) \neq \emptyset \bigr\}.
    \]
    \begin{itemize}
    \item If $\widetilde E \neq \emptyset$, then the answer is
      \emph{yes}, there exists an epimorphism~$\Gamma \longrightarrow
      F$.
    \item If $\widetilde E = \emptyset$, the answer is \emph{no}.
    \end{itemize}
  \end{itemize}
  Correctness of this algorithm is guaranteed by
  Corollary~\ref{cor:concreteepi}.
\end{proof}


\medskip
\vfill

\noindent
\emph{Stefan Friedl}\\
\emph{Clara L\"oh}\\[.5em]
  {\small
  \begin{tabular}{@{\qquad}l}
    Fakult\"at f\"ur Mathematik,
    Universit\"at Regensburg,
    93040 Regensburg\\
    \textsf{stefan.friedl@mathematik.uni-r.de}, 
    \textsf{http://www.mathematik.uni-r.de/friedl}\\
    \textsf{clara.loeh@mathematik.uni-r.de}, 
    \textsf{http://www.mathematik.uni-r.de/loeh}
  \end{tabular}}

\end{document}